\documentclass[a4paper,11pt]{amsart}

\usepackage{amsfonts,amssymb}
\usepackage{textcomp}

\theoremstyle{plain}

\newtheorem{thm}{Theorem}[section]
\newtheorem{dfn}[thm]{Definition}
\newtheorem{pro}[thm]{Proposition}
\newtheorem{lem}[thm]{Lemma}
\newtheorem{cor}[thm]{Corollary}

\newtheorem{rmk}[thm]{Remark}

\newcommand{\N}{\mathbb{N}}

\begin{document}

\title[Non-inner automorphisms of order $p$ in NC $p$-groups ]%
{Non-inner automorphisms of order $p$ in finite normally constrained $p$-groups }

\author[N.Gavioli]{Norberto Gavioli}
\address{DISIM \\Universit\`a degli studi dell'Aquila\\ 67100 L'Aquila, Italy\\
{\it E-mail address}: { \tt gavioli@univaq.it }}

\author[L.Legarreta]{Leire Legarreta}
\address{Matematika Saila\\ Euskal Herriko Unibertsitatea UPV/EHU\\
48080 Bilbao, Spain\\ {\it E-mail address}: {\tt leire.legarreta@ehu.eus}}

\author[M.Ruscitti]{Marco Ruscitti}
\address{DISIM \\Universit\`a degli studi dell'Aquila\\ 67100 L'Aquila, Italy\\
{\it E-mail address}: { \tt marco.ruscitti@dm.univaq.it }}

\thanks{ The second author is supported by the Spanish Government, grants MTM2011-28229-C02-02 and MTM2014-53810-C2-2-P, and by the Basque Government, grant IT753-13 and IT974-16. The third author would like to thank the Department of Mathematics at the University of the Basque Country for its excellent hospitality while part of this paper was being written.}

\keywords{Finite p-groups, non-inner automorphisms, derivation, 
thin p-groups, normally constrained p-groups.\vspace{3pt}}

\begin{abstract}
In this paper we study the existence of at least one non-inner automorphism of order $p$ in a finite normally constrained $p$-group when $p$ is an odd prime. 
\end{abstract}

\maketitle

\author{Norberto Gavioli}
\address{DISIM \\Universit\`a degli studi dell'Aquila\\ 67100 L'Aquila, Italy\\
{\it E-mail address}: { \tt gavioli@univaq.it}}

\vspace{8pt}

\author{Leire Legarreta}
\address{Matematika Saila\\ Euskal Herriko Unibertsitatea UPV/EHU\\
48080 Bilbao, Spain\\ {\it E-mail address}: {\tt leire.legarreta@ehu.eus}}

\vspace{8pt}

\author{Marco Ruscitti}
\address{DISIM \\Universit\`a degli studi dell'Aquila\\ 67100 L'Aquila, Italy\\
{\it E-mail address}: { \tt marco.ruscitti@dm.univaq.it }}

\section{Introduction}\label{sec:introduction}

The main goal of this paper is to contribute to the longstanding conjecture of Berkovich posed in 1973, that conjectures that every finite $p$-group admits a non-inner automorphism of order $p$, where $p$ denotes a prime number \cite[Problem 4.13]{khukhro:2010}. The conjecture has attracted the attention of many mathematicians during the last couple of decades, and has been confirmed for many classes of finite $p$-groups. It is remarkable to put on record that, in 1965, Liebeck \cite{liebeck:1965} proved the existence of a non-inner automorphism of order $p$ in all finite $p$-groups of class $2$, where $p$ is an odd prime. However, the fact that there always exists a non-inner automorphism of order $2$ in all finite $2$-groups of class $2$ was proved by Abdollahi \cite{abdollahi:2007} in 2007. The conjecture was confirmed for finite regular $p$-groups by Schmid \cite{schmid:1980} in 1980. Indeed, Deaconescu \cite{deaconescu:2002} proved it for all finite $p$-groups $G$ which 
 are not strongly Frattinian. Moreover, Abdollahi \cite{abdollahi:2010} proved it for finite $p$-groups $G$ such that $G/Z(G)$ is a powerful $p$-group, and Jamali and Visesh \cite{jamali:2013} did the same for finite $p$-groups with cyclic commutator subgroup. In the realm of finite groups, quite recently, the result has been confirmed  for semi-abelian $p$-groups by Benmoussa and  Guerboussa \cite{benmoussa:2015}, and for $p$-groups of nilpotency class $3$, by Abdollahi, Ghoraishi and Wilkens  \cite{abdollahi':2013}. To be more precise, Abdollahi and Ghoraishi in \cite{abdollahi13} proved that in some cases the non-inner automorphism of order $p$ can be chosen so that it leaves $Z(G)$ elementwise fixed. Finally, Abdollahi et al \cite{abdollahi:2014} proved the conjecture for $p$-groups of coclass $2$, and quite recently in \cite{ruscitti}  M.Ruscitti, L. Legarreta and M.K.Yadav did the same for $p$-groups of coclass $3$ when $p$ is a prime different from $3$.

With the contribution of this paper we add bassically another class of finite $p$-groups to the above list, by proving that the above mentioned conjecture holds true for all finite normally constrained $p$-groups when $p$ is an odd prime.

The organization of the paper is as follows. In Section 2 we exhibit some preliminary facts and tools that will be used in the proofs of the main results of the paper and we introduce the family of normally constrained $p$-groups. In Section 3 we prove that two generator normally constrained $p$-groups have a non-inner automorphism of order $p$. In Section 4 we extend these results to groups generated by more than two elements. Throughout the paper $p$ will be an odd prime, since there are no examples or results in normally constrained $2$-groups, except for two generator normally constrained $2$-groups that are, as we will see, $2$-groups of maximal class.

\noindent
Throughout the paper, most of the notation is standard and it can be found, for instance, in \cite{Ro}.

\section{Preliminaries}
Let us start this section recalling some facts about derivations, and some related lemmas, which will be useful to prove the main Theorem \ref{main} and Theorem \ref{cons} of the paper. The reader could be referred to  \cite{gavioli:1999} for more details and explicit proofs about derivations.

\begin{dfn}
Let $G$ be a group and let $M$ be a right $G$-module. A derivation $\delta:G \rightarrow M$ is a function such that $$\delta(gh)={\delta(g)}^h\delta(h),   \text{ for all  } g, h \in G.$$
\end{dfn}

In terms of its properties, it is well-known that a derivation is uniquely determined by its values over a set of generators of $G$. Let $F$ be a free group generated by a finite subset $X$ and let $G=\langle X : r_{1}, \ldots, r_{n} \rangle$ be a group whose free presentation is $F/R$, where $R$ is the normal closure of the set of relations $\{ r_{1} , \ldots , r_{n}\}$ of $G$. Then a standard argument shows that $M$ is a $G$-module if and only if $M$ is an $F$-module on which $R$ acts trivially. Indeed, if we denote by $\pi$ the canonical homomorphism $\pi: F \rightarrow G$, then the action of $F$ on $M$ is given by $mf=m\pi(f)$, for all $m\in M$ and all $f\in F$. 
Continuing with the same notation, we have the following results.

\begin{lem} \label{determineunique}
Let $M$ be an $F$-module. Then every function $f: X \rightarrow M$ extends in a unique way to a derivation $\delta: F \rightarrow M$.
\end{lem}

\begin{lem} \label{Der}
Let $M$ be a $G$-module and let $\delta: G \rightarrow M$ be a derivation. Then $\bar{\delta} : F \to M$ given by the composition $\overline{\delta}(f)=\delta(\pi(f))$ is a derivation  such that $\overline{\delta}(r_i)=0$ for all $i\in \{1,\ldots,n\}$. Conversely, if $\overline{\delta}:F \rightarrow M$ is a derivation such that $\overline{\delta}(r_i)=0$ for all $i \in \{1,\ldots, n\}$, then $\delta(fR)=\overline{\delta}(f)$ defines, uniquely, a derivation on $G=F/R$ to $M$ such that $\overline{\delta}=\delta \circ \pi$.
\end{lem}

In the following lemma, we study the relationship between derivations and automorphisms of a finite $p$-group. 

\begin{lem} 
\label{lift}
Let $G$ be a finite $p$-group and let $M$ be a normal abelian subgroup of $G$ viewed as  a $G$-module. Then for any derivation $\delta: G \rightarrow M$, we can define uniquely an endomorphism $\phi$ of $G$ such that $\phi(g)=g\delta (g)$ for all $g \in G$. Furthermore, if  $\delta(M)=1$,  then $\phi$ is  an automorphism of $G$.
\end{lem}

In order to reduce some calculations in terms of commutators, we keep in mind the following result.

\begin{lem} \label{free}
Let $F$ be a free group, $p$ be a prime number and $A$ be an $F$-module. If $\delta : F \rightarrow A$ is a derivation then, 
\begin{enumerate}
\item  $\delta(F^{p})=\delta(F)^{p}[\delta(F),_{p-1}F]$,
\item if $[A,_{i}F]=1$, we have $\delta(\gamma_{i}(F)) \leq [\delta(F), _{i-1}F]$ for all $ i \in \mathbb{N}$.
\end{enumerate}
\end{lem}

\begin{proof}
Let $ x \in F$. We have $\delta(x^{p})=\delta(x)^{x^{p-1} + x^{p-2} + \ldots + 1 }$. Since $(x-1)^{p-1} \equiv  x^{p-1} + x^{p-2} + \ldots + 1\mod p$, the first assertion follows. Now let us prove the second assertion by induction on $i$.
 Clearly, the assertion holds when $i=1$. 
Let us suppose, by inductive hypothesis that if $[A,_{k}F]=1$ 
then $\delta(\gamma_{k}(F)) \leq [\delta(F), _{k-1}F]$ for some $ k \in \mathbb{N}$. 
Let us take any $a \in F$ and any $b \in \gamma_{k}(F)$, and let us suppose 
that $[A,_{k+1}F]=1$. Then $$\delta([a,b])=[\delta(a),b][a,\delta(b)][a,b,\delta(a)][a,b,\delta(b)] \in 
  [\delta(F), _{k}F].$$
\end{proof}

Now let us introduce the family of normally constrained $p$-groups, according to \cite{bonmassar:1999}.

\begin{dfn}
Let $G$ be a finite $p$-group and let $c(G)$ be its nilpotency class. We say that $G$ is normally constrained ($NC$ for short) if for
every $i=1,\ldots, c(G)$ the following equivalent conditions hold true:
\begin{enumerate}
\item $\gamma_{i}(G)$ is unique of its order, for all $i=1,\ldots, c(G)$,
\item if $N \lhd G$ then $N \leq \gamma_{i}(G)$ or $N \geq \gamma_{i}(G)$, for all $i=1,\ldots, c(G)$,
\item if $x \in G - \gamma_{i}(G)$ then $\gamma_{i}(G) \leq \langle x \rangle^{G}.$
\end{enumerate}
\end{dfn} 

Let us note that factor groups of $NC$-$p$-groups are $NC$, and that the second statement is equivalent to say that if $N \lhd G$ then there exists a positive integer $i$ such that $\gamma_{i}(G) \geq N \geq \gamma_{i+1}(G)$. Now let us list useful properties of these groups whose proofs can be found in \cite{bonmassar:1999}.

\begin{pro} \label{bumi3.1}
Let $G$ be an $NC$-$p$-group of nilpotency class at least $3$. Then $\overline{G}=G/\gamma_{3}(G)$ is a special $p$-group (i.e. $Z(\overline{G}) =\Phi(\overline{G})=\gamma_{2}(\overline{G})$) of exponent $p$ and $|\gamma_{2}(G) / \gamma_{3}(G)| ^{2} =|G/\gamma_{2}(G)|$.
\end{pro} 

Repeating the argument of the previous result it is possible to prove the following corollary.

\begin{cor}\label{bumi3.2}
Let $G$ be a $NC$-$p$-group of nilpotency class at least $3$. Then $\gamma_{i}(G)/\gamma_{i+2}(G)$ is elementary abelian for all $i \geq 2$.
\end{cor}

Thus every $NC$-$p$-group is generated by an even number of elements, and every double quotient of its lower central series has exponent $p$. Moreover we can obtain a stronger property of such groups, known as \emph{covering property}.

\begin{pro} \label{bumi3.3}
Let $G$ be a $p$-group of nilpotency class greater than or equal to $3$. The following conditions are
equivalent:
\begin{enumerate}
\item $G$ is a $NC$-$p$-group,
\item for all $i \geq 1$ and for all $x \in \gamma_{i}(G) - \gamma_{i+1}(G)$ we have  $[x,G]\gamma_{i+2}(G)=\gamma_{i+1}(G)$.
\end{enumerate}
\end{pro}

\begin{cor} \label{bumi3.4}
Let $G$ be a $NC$-$p$-group of nilpotency class at least $3$. Then the upper and lower central series of $G$ coincide. 
\end{cor}

The last result implies, clearly, that in a normally constrained $p$-group $G$ the quotients of the terms of the lower central series are elementary abelian, and so $\gamma_{2}(G)=\Phi(G)$. Let us finish this section recalling the following result. 

\begin{thm} \label{bumi3.5}
Let $G$ be a $NC$-$p$-group of nilpotency class $c(G) \geq 3$ such that $|G:\gamma_{2}(G)|=p^{2n}$ for some $n \in \N$. Then for all $1 \leq i < c(G) $ we have that $p^{n} \leq |\gamma_{i}(G):\gamma_{i+1}(G)| 	\leq p^{2n}$. 
\end{thm}

\section{Berkovich Conjecture for two generator normally constrained $p$-groups}

To develop this section, let us start introducing the family of thin $p$-groups. Firstly, let us recall that in a group $G$ an antichain is a set of mutually incomparable elements in the lattice of its normal subgroups.  It is well-known that, if $G$ is a $p$-group of maximal class, then the lattice of its normal subgroups consists of $p+1$ maximal subgroups and of the terms of the lower central series of $G$. Thus, a $p$-group of maximal class has only one antichain, which consists of its maximal subgroups. The necessity to extend the family of groups of maximal class to a bigger family of $p$-groups with a bound on the antichains, leads us to introduce the formal definition of thin $p$-group. Let us introduce the definition of thin $p$-groups as in \cite{caranti:1996}. 

\begin{dfn}
Let $G$ be a finite $p$-group. Then $G$ is thin if every antichain in $G$ contains at most $p+1$ elements.
\end{dfn}

The following results about finite thin $p$-groups are discussed in \cite{brandl:1992}.
\begin{lem} \label{basics}
Let $G$ be a finite thin $p$-group, and let $p$ be an odd prime. 
If $N$ is a normal subgroup of $G$,  then $N$ is a term of the lower central series of $G$ if and only if $N$ is the unique normal subgroup of its order.
\end{lem}

\begin{rmk} \label{thinnc}
Let $G$ be a  two generator normally constrained $p$-group. According to Theorem \ref{bumi3.5} all quotients of the terms of the lower central series of $G$ are of order at most $p^{2}$, and by Lemma \ref{basics} all terms of the lower central series are unique in their order. Thus every antichain in $G$ cannot contain more than $p+1$ elements, so $G$ is thin. On the other hand, since an elementary abelian $p$-group is thin if and only if  its order  is $p^{2}$ (see \cite{caranti:1996}), then every finite thin $p$-group is a two generator group. Clearly, if $G$ is thin, then $G$ is normally constrained, so thin $p$-groups are exactly normally constrained two generator $p$-groups. 
\end{rmk}

Forwards, we again mention some already proved results about the existence of non-inner automorphisms of order $p$ in certain specific cases to avoid from now onwards repetitions, not only in the case of finite thin $p$-groups, but also in the case of finite normally constrained $p$-groups (which will be studied in the fourth section).
 
Since Liebeck in \cite{liebeck:1965} proved the existence of at least a non-inner automorphism of order $p$ in all finite $p$-groups of class $2$ for any $p$ odd prime,  Abdollahi in \cite{abdollahi:2007} proved the existence of such an non-inner automorphism of order $2$ in all finite $2$-groups of class $2$,  and Abdollahi, Ghoraishi and Wilkens in  \cite{abdollahi':2013} did the same in the case of finite $p$-groups of nilpotency class $3$, from now onwards in our study we will deal with finite $p$-groups of nilpotency class $c=c(G)\geq 4$. On the other hand,  since Deaconescu in \cite{deaconescu:2002} proved the existence of at least a non-inner automorphism of order $p$ for all finite $p$-groups which are not strongly Frattinian, we may assume that the finite $p$-groups $G$ we are interested in, are strongly Frattinian, in other words, that the groups of our interest satisfy  $C_{G}(\Phi(G))=Z(\Phi(G))$. 

Furthermore, as a result due to Abdollahi in \cite{abdollahi:2010}, we know that if $G$
  is a finite $p$-group such that $G$ has no non-inner automorphisms of order $p$ leaving $\Phi(G)$ elementwise fixed, then $d(Z_{2}(G)/Z(G))=d(G)d(Z(G))$. Thus, in view of this previous matter, we may assume that the condition $d(Z_{2}(G)/Z(G)) = d(G)d(Z(G))$ holds.

\begin{rmk} \label{structurethin}
By Corollary \ref{bumi3.4} and Theorem \ref{bumi3.5} the lower and the upper central series of thin $p$-groups coincide and  all quotients of these series are elementary abelian $p$-groups of order at most $p^{2}$. Moreover, $Z(G)$ must be cyclic of order $p$, since we assume that the condition $d(Z_{2}(G)/Z(G)) = d(G)d(Z(G))$ holds. In particular, the quotients of the lower and upper central of finite thin $p$-groups have exponent $p$.  Moreover, if $G$ is a finite thin $p$-group, then $\Phi(G)=\gamma_{2}(G)$, $Z_2(G)\leq Z(\Phi(G))$, and the property stated in the second part of Proposition \ref{bumi3.3} holds equivalently for terms of the upper central series of $G$. In addition to this, since in \cite{brandl:1988} (see Theorem $B$) it is shown that every finite thin $2$-group is of maximal class, in the following we just focus our attention on finite thin $p$-groups, where $p$ is an odd prime.
\end{rmk} 

Now we are ready to prove the next theorem.

\begin{thm} \label{main}
Let $G$ be a finite thin $p$-group, where  $p$ is an odd prime. Then $G$ has a non-inner automorphism of order $p$.
\end{thm}
\begin{proof}
Let  us denote $c$ the nilpotency class of $G$ with $c\geq 4$. From the assumptions and the consequences of Remarks \ref{thinnc} and \ref{structurethin} we know that $G$ is a two generator $p$-group, the lower and the upper central series of $G$ coincide,  $\Phi(G)=\gamma_{2}(G)$, $Z_2(G)\leq Z(\Phi(G))$, $Z(G) \cong C_{p}$, $d(Z_{2}(G)/Z(G)))=d(G)d(Z(G))$ and  $Z_{2}(G)/Z(G) \cong C_{p} \times C_{p}$. In particular, $[Z_{2}(G), \gamma_{2}(G)]=1$ and $\Omega_{1}(Z_{2}(G))$ is an elementary abelian subgroup of $G$. 

Indeed, $G/\gamma_{3}(G)$ has order $p^{3}$, class $2$ and  by Proposition \ref{bumi3.1} we may assume that its exponent is $p$, i.e. $G/ \gamma_{3}(G)$ is an extraspecial $p$-group of exponent $p$ and order $p^{3}$.   

Our goal is to obtain at least an automorphism of $G$ of order $p$. To do that,  firstly, we define an assignment  on generators of the free group generated by two elements sending them to  $\Omega_{1}(Z_{2}(G))$. By Lemma \ref{determineunique} it is possible to extend these assignments to a derivation. Secondly, we show that this map preserves the relations defining the quotient $G/ \gamma_{3}(G)$, and then we apply Lemma \ref{Der} to induce a derivation from  $G/ \gamma_{3}(G)$ to $\Omega_{1}(Z_{2}(G))$. And finally, we lift this found map to a derivation from $G$ to $\Omega_{1}(Z_{2}(G))$ applying Lemma \ref{lift}. In the following paragraphs we describe in detail each of these mentioned steps.
\noindent

To begin with, let $x \rightarrow u, y \rightarrow v $ be an assignment on generators $x,y$ of the two generator free group $F_{2}$, with $u,v \in \Omega_{1}(Z_{2}(G))$. By Lemma \ref{Der} this assignment extends uniquely to a derivation $ \delta : F_{2} \rightarrow \Omega_{1}(Z_{2}(G))$ such that $\delta(x)=u$ and $\delta(y)=v$. Doing some abuse of notation, we  assume that $G/ \gamma_{3}(G)$ corresponds to the presentation  $ \langle x,y  \ | \ x^{p}, y^{p}, [y,x,x], [y,x,y] \rangle$. Next, let us see that the equalities $\delta(x^{p})=1$, $\delta(y^{p})=1$, $\delta([y,x,x])=1$ and $\delta([y,x,y])=1$ hold. To do it, firstly, let us show that $\delta$ is trivial on the $p$-th powers of elements of $F_2$. In fact, considering $\pi'$ the canonical epimorphism from $F_{2}$ to $G/\gamma_{3}(G)$ we have
$$\delta(f^{p})= \delta(f)^{f^{p-1}}\delta(f^{p-1})= \cdots = \delta(f)^{f^{p-1}+ \cdots + 1}=\delta(f)^{\pi'(f^{p-1}+ \cdots + 1)}=$$
$$= \delta(f)^{\pi'(f)^{p-1}+ \cdots + 1}=\delta(f)^{p}[\delta(f),\pi'(f)]^{\binom{p}{2}}=1, \quad{ for \ all \ } f \in F_2.$$
Secondly, let us analyze the behaviour of $\delta$ on commutators. Since for any $g,h \in G$ it holds that $gh=hg[g,h]$, then applying $\delta$ to this previous equality, we get  $\delta(gh)=\delta(hg[g,h])=\delta(hg)^{[g,h]}\delta([g,h])$, and consequently taking into account as well that $[Z_{2}(G), \gamma_{2}(G)]=1$, we get
$\delta([g,h])=\delta(gh) (\delta(hg)^{[g,h]})^{-1}=\delta(gh) (\delta(hg))^{-1}=[\delta(g),h][g,\delta(h)]$, which is an element of $Z(G)$.

Moreover,
$$  \delta([y,x,y])=[\delta([y,x]),y] [[y,x],\delta(y)]=1,$$    $$\delta([y,x,x])=[\delta([y,x]),x] [[y,x],\delta(x)]=1.$$

\vspace{5pt}

Let us underline that the two previous equalities can be obtained as well, applying properly item (ii) of Lemma \ref{free}.

\noindent

Now by Lemma \ref{Der} we can induce a derivation $\overline{\delta}$ from  $G/ \gamma_{3}(G)$ to $\Omega_{1}(Z_{2}(G))$. The map $\delta' :G \rightarrow \Omega_{1}(Z_{2}(G))$ defined for all $ g \in G$ by the law $\delta ' (g)=\overline{\delta}(g\gamma_{3}(G))$  is a derivation from $G$ to $ \Omega_{1}(Z_{2}(G))$.  By Lemma \ref{lift}, $\delta'$ induces an automorphism $\phi$ of $G$ by the law $\phi(g)=g\delta'(g)$ for all $g \in G$, leaving, in particular, $\Phi(G)$ elementwise fixed. Clearly, $Z_{2}(G) \leq Z(\gamma_{3}(G))$. In fact, this previous inclusion holds since $Z_{2}(G) \leq Z(\Phi(G))$, $\Phi(G)=\gamma_{2}(G)$ and $Z_{2}(G) = \gamma_{c-1}(G) \leq \gamma_{3}(G)$. Thus this allows us to prove that the automorphism $\phi$ has order $p$. In fact,  
$\phi^{p}(g)=\phi^{p-1}(g\delta'(g))=\phi^{p-1}(g) \phi^{p-1}(\delta'(g))=\phi^{p-1}(g) \delta'(g)=\phi^{p-2}(g\delta'(g))\delta'(g)=\cdots= g (\delta'(g))^{p}=g$, for all $g \in G$.

\noindent

Repeating this previous construction we produce a set of automorphisms of $G$ of order $p$, whose size is equal to $|\Omega_{1}(Z_{2}(G))|^{2}$, in other words, whose size is the number of possible choices for the images of the above generators. Next we distinguish the only two possible cases: $Z_{2}(G) \cong C_{p} \times C_{p} \times C_{p}$ or  $Z_{2}(G) \cong C_{p^{2}} \times C_{p}$. In the former case, when  $Z_{2}(G) \cong C_{p} \times C_{p} \times C_{p}$, we can produce $p^{6}$ automorphisms of $G$ of order $p$. However, the number of inner automorphisms of $G$ induced by elements of $Z_{3}(G)$ is at most $p^{4}$.  Thus, a simple counting argument is enough to say that $G$ has a non-inner automorphism of $G$ of order $p$, and we get the statement of the Theorem in this case. Otherwise, in the latter case, when $Z_{2}(G) \cong C_{p^{2}} \times C_{p}$, let us choose an assignment $x \rightarrow u, y \rightarrow v $ such that $u,v \in \Omega_{1}(Z_{2}(G)) - \{1\}$, $u$ is n
 ot central, and
  let $\phi$ be the automorphism of $G$ of order $p$ obtained by this assignment. On the other hand, we know that $\phi$ is inner if and only if there exists an element $h_{\phi}$ of $Z_{3}(G) - Z_{2}(G)$ such that $h_{\phi}^{p} \in Z(G)$, $\phi(g)=g^{h_{\phi}}$ and $\delta'(g)=[g,h_{\phi}]$ for all $g \in G$. Under these circumstances, since $Z(G) \leq \Omega_{1}(Z_{2}(G))$ we deduce that  $[h_{\phi}, G]Z(G) \leq \Omega_{1}(Z_{2}(G))  < Z_{2}(G)$,  which is in contradiction with the analogous coverty property of Proposition \ref{bumi3.3}. Consequently, this second case does not happen and the statement of the Theorem is proved.
\end{proof}

\section{Berkovich's conjecture for normally constrained $p$-groups}

In this section we attack Berkovich's conjecture in normally constrained $p$-groups, for groups with more than $2$ generators. In terms of the following remark, let us see that we can deal with normally constrained  $p$-groups with cyclic center.

\begin{rmk} \label{secondcenter}
As we have established in the previous section, we are dealing with finite $p$-groups $G$ such that $d(Z_{2}(G)/Z(G)) = d(G)d(Z(G))$, so we can assume by Theorem \ref{bumi3.5} that the center of $G$ is cyclic of order $p$ (otherwise the quotient $Z_{2}(G)/Z(G)$ would contradict that above theorem) and that $d(Z_{2}(G)/Z(G)) = d(G)$. Moreover, from \cite{deaconescu:2002} we can assume as well that the second center $Z_{2}(G)$ is abelian, since otherwise the problem is already solved. 
\end{rmk}

\begin{thm}
\label{cons}
Let $G$ be a normally constrained finite $p$-group generated by more than two elements. Then $G$ has a non-inner automorphism of order $p$.
\begin{proof}
Let $G$ be a normally constrained $p$-group $d$-generated by the elements $x_{1},\ldots,x_{d}$, with $d\geq 3$. Thus according to previous Remark \ref{secondcenter}, the elementary abelian quotient group $Z_{2}(G)/Z(G)$   is generated by $d$ elements, and $Z_{2}(G)$ is an elementary abelian subgroup of $G$ generated by $d+1$ elements.

Let us note that the assignment $ x\Phi(G) \mapsto [\cdot, x] $ defines an injective homomorphism $\phi$ from $G/\Phi(G)$ into $Hom(Z_{2}(G)/Z(G),Z(G))$, which is actually an isomorphism since $Z_{2}(G)/Z(G)$ is generated by $d$ elements. Indeed the elements $\phi(x_{i})$, where $i=1, \ldots, d$, form a basis of the dual space of $Z_{2}(G)/Z(G)$. This implies that the intersection of the kernels of $d-1$ of such maps is a vector space one dimensional. Moreover, the order of $K:=\bigcap_{i=2}^{d}C_{Z_{2}(G)}(x_{i})$ is exactly $p^{2}$. Now let us take $u_{1} \in K - Z(G)$ and let us define the following assignments: 

\[ \delta = \left\{ \begin{array}{ll}
         x_{1} \rightarrow u_{1} \\
         x_{i}=1 \ for \ all \ i \geq 2 \\ 
         \end{array} \right. \]

This $\delta$ map is a derivation from the $d$-generator free group $F_{d}$ to the $F_{d}$-module $K$, and, arguing as in Theorem \ref{main}, this derivation induces a new derivation from $$G/\Phi(G) = \langle x_{1},\ldots, x_{d} \ | \ x_{1}^{p}, \ldots, x_{d}^{p}, [x_{i},x_{j}] \ for \ all \ i,j \rangle $$ to $Z_{2}(G)$, since for all $i,j\in \{1, \ldots,d\}$ 
$$\delta(x_{i}^{p})=\delta(x_{i})^{p}[\delta(x_{i}),x_{i}]^{\binom{p}{2}} =1$$ and $$\delta([x_{i},x_{j}])=[x_{i},\delta(x_{j})][\delta(x_{i}),x_{j}]=1.$$ 
Let us note that the previous equalities hold by definition of $K$. 
Next, since $K \leq Z(\Phi(G))$, this obtained derivation extends by Lemma \ref{lift} to an automorphism $\phi$, defined by $\phi(g)=g\delta(g)$ for all $g \in G$, whose order is $p$ and that leaves the $Frattini$ subgroup elementwise fixed. Moreover, if there exists $h \in Z_{3}(G)-Z_{2}(G)$ such that $\phi(g)=g^{h}=g[g,h]$ for all $g \in G$, then it would deduce that $[h,G]=K$, but this jointly with the fact that $Z(G) < K < Z_{2}(G)$ contradicts the covering property of $G$, listed in Proposition \ref{bumi3.3}. As a consequence, the obtained automorphism of order $p$ is non-inner, as desired.
\end{proof}
\end{thm}

\begin{rmk}
Let us note that the proof of the above Theorem \ref{cons} actually works for two generator groups as well. However, it has been a wish of the authors to present to the reader two different ways of producing derivations and solving the problem in the case of two generator normally constrained-$p$ groups, as it has been done in Theorems \ref{main} and \ref{cons}, respectively.
\end{rmk}
 
\begin{rmk}
In conclusion, for every odd prime number $p$, normally constrained $p$-groups have at least a non-inner automorphism of order $p$. Moreover, since thin $2$-groups are of maximal class, it is well known that they have a non-inner automorphism of order $2$. But what happens in the case of normally constrained $2$-groups with more than two generators and nilpotency class greater than or equal to $4$? The existing bibliography give us neither examples nor results about such groups.
\end{rmk}

\section{Acknowledgements}
The third author would like to thank the Department of Mathematics at the University of the Basque Country for its excellent hospitality while part of this paper was being written; he also wish to thank Professors Gustavo A. Fern\'andez Alcober and Carlo Maria Scoppola for their suggestions.

\bibliography{articles}

\bigskip
\bigskip

\end{document}